\documentclass{amsart}
\usepackage{natbib, amssymb,latexsym, amscd}
\usepackage[all]{xy}
\usepackage{graphicx}
\usepackage{mathrsfs}
\usepackage{amsmath}

 \newtheorem{theorem}{Theorem}[section]
 \newtheorem{cor}[theorem]{Corollary}
 \newtheorem{lemma}[theorem]{Lemma}
 \newtheorem{proposition}[theorem]{Proposition} \theoremstyle{definition}
 \newtheorem{definition}[theorem]{Definition}
 \theoremstyle{definition}
 \newtheorem{example}[theorem]{Example}
 \theoremstyle{remark}
 \newtheorem{rem}[theorem]{Remark}
 \numberwithin{equation}{section}

\newcommand{\ben}{\begin{equation}}
\newcommand{\een}{\end{equation}}

\newcommand{\integer}{\ensuremath{{\mathbb Z}}}

\newcommand{\real}{\ensuremath{{\mathbb R}}}
\newcommand{\complex}{\ensuremath{{\mathbb C}}}

\newcommand{\LL}{\mathcal{L}}

\newcommand{\HH}{\mathcal{H}}

\newcommand{\TT}{\mathcal{T}}

\newcommand{\To}{\longrightarrow}

\newcommand{\gr}{\mathfrak}
\newcommand{\GGC}{{\mathscr{G}}}
\newcommand{\XXC}{{\mathscr{X}}}

\newcommand{\TTB}{{\mathbb{T}}}
\newcommand{\JJB}{{\mathbb{J}}}

\begin{document}

\title{Chern-Weil homomorphism in twisted equivariant cohomology}

\author{Alexander Caviedes,  Shengda Hu and Bernardo Uribe}
\thanks{ The first and the third author were partially supported by COLCIENCIAS. }
\address{Departamento de Matem\'{a}ticas, Universidad Nacional,
Ciudad Universitaria, Bogot\'a, COLOMBIA}
\email{alephxander@gmail.com}
\address{D\'epartement de Math\'ematiques et de Statistique, Universit\'e de  Montr\'eal, CP 6128 succ Centre-Ville, Montr\'eal, QC H3C 3J7, Canada}
\address{Department of Pure Mathematics, University of Waterloo, 200 University Avenue West, Waterloo, Ontario, Canada N2L 3G1, Canada}
\email{shengda@dms.umontreal.ca, hshengda@math.uwaterloo.ca}
\address{Departamento de Matem\'{a}ticas, CINVESTAV, Apartado Postal 14-740 07000 M\'exico
D.F. M\'EXICO}
\address{Departamento de Matem\'{a}ticas, Universidad de los Andes,
Carrera 1 N. 18A - 10, Bogot\'a, COLOMBIA}
 \email{buribe@uniandes.edu.co, buribe@math.cinvestav.mx }

\begin{abstract}
We describe the Cartan and Weil models of twisted equivariant cohomology 
together with the Cartan homomorphism among the two, and we extend the Chern-Weil 
homomorphism to the twisted equivariant cohomology. We clarify that in order to have a cohomology theory,
 the coefficients of the twisted equivariant cohomology must be taken in
the completed polynomial algebra over the dual Lie algebra of $G$. We recall the relation between
the equivariant cohomology of exact Courant algebroids and the twisted equivariant cohomology, and we show
how to endow with a generalized complex structure the finite dimensional approximations of the Borel construction $M\times_G EG_k$, whenever
the generalized complex manifold $M$ possesses a Hamiltonian $G$ action.
\end{abstract}

\maketitle

\section{Introduction}
In ordinary equivariant cohomology there are two well known models, the
bigger while more geometrical Weil model and the smaller
while more algebraic Cartan model. In \cite{HuUribe}, the authors
defined twisted equivariant cohomology following the Cartan model and
showed that the twisted equivariant cohomology satisfies the cohomology
axioms under some natural assumptions, for example, that the group $G$ is a compact Lie group. 
Under these assumptions, the
corresponding Cartan model is twisted by a closed equivariant $3$-form
 (\S\ref{compact}). In the non-twisted case, the Cartan homomorphism  gives
a quasi-isomorphism between the Weil and the Cartan models. In this
article, we show (\S\ref{twisted}, \S\ref{chern}) that in the twisted  case,
the twisted equivariant theory also has the corresponding twisted
Weil model and a that the Cartan homomorphism could be extended to the twisted models yielding also
a quasi-isomorphism between them  . Moreover, we show that the Chern-Weil homomorphism
can also be extended to the twisted equivariant case, which we use to demonstrate that the 
twisted equivariant cohomology is isomorphic to the twisted cohomology of the Borel construction; this provides us with an
alternative proof of the fact that the twisted cohomology is indeed a cohomology theory (cf.  \cite{HuUribe}).

One subtlety comes up when defining the equivariant theory
in the twisted case. In the ordinary Cartan model, we consider the
complex $ (\Omega^*(M) \otimes S(\gr g^*))^G$ with the equivariant
differential $d_G$. In the twisted case we have to consider the complex
$ (\Omega^*(M) \otimes \widehat{S}(\gr g^*))^G$ where the $\widehat{S}(\gr g^*)$ is the completion of the algebra
$S(\gr g^*)$, because otherwise, the twisted cohomology defined over the uncompleted algebra could  be non finitely generated
as we show in the Appendix (\S\ref{appendix}).

The twisted equivariant cohomology appeared as the equivariant theory
of generalized complex manifolds \cite{HuUribe, Lin2, Lin3, Lin4}. Then in the case of a Hamiltonian $G$-action on a generalized complex 
manifold $M$, when $G$ is compact and the finite approximations 
$BG_k$ of $BG$ are symplectic, 
we apply the coupling construction using the principle $G$-bundle $EG_k \to BG_k$ to
show that the finite approximations of the Borel construction $M \times_G EG_k$
are generalized complex as well. The explicit computation of the twisting form here
coincides with the result given by the twisted version of the Chern-Weil homomorphism.

The structure of the article is the following. In \S\ref{Courant} we recall the 
basics on the symmetries of the exact Courant algebroids as well as the definition of extended
equivariant cohomology from \cite{HuUribe}. In \S\ref{twisted} we recall the twisted 
equivariant cohomology in both the Cartan and the Weil models. We also clarify the subtlety 
about the completion mentioned above. Then we show in \S\ref{compact}, that when the Lie group $G$
is compact the extended equivariant cohomology is isomorphic to some twisted equivariant
cohomology. The Chern-Weil homomorphism and its consequences are shown in \S\ref{chern}.
In \S\ref{hamilton}, we describe the coupling construction for Hamiltonian actions on 
generalized complex manifolds using principle $G$-bundles. The Appendix \S\ref{appendix}
 compares the cohomologies of the completed and uncompleted complexes for the twisted equivariant cohomology
  using the example
\ref{Example circle circle}.

\vspace{0.3cm}

 {\bf Acknowledgements}

 S. Hu would like to thank the hospitality of Universidad de Los Andes where part of this work was done. B. Uribe would like to thank
 the hospitality of the Erwin Schr\"odinger Institut in Vienna and the Max Planck Institut f\"ur Mathematik in Bonn where part is this work was also done.
 We would like to thank conversations with E. Backelin, H. Bursztyn, G. Cavalcanti, T. Str\"obl and M. Xicot\'{e}ncatl. Finally, we would like to specially thank H. Bursztyn for the comments and advices he gave us after reading a preprint of this paper.

\section{Symmetries of Exact Courant Algebroids} \label{Courant}

\subsection{Exact Courant algebroid}
An exact Courant algebroid  is a Courant algebroid $\TT M$ over a manifold $M$ that
fits into the exact sequence
$$\xymatrix{
0 \ar[r] & T^*M \ar[r] & \TT M \ar[r]^a & TM \ar[r] & 0. }$$

The Courant algebroid $\TT M$ is endowed with a nondegenerate
symmetric bilinear form and a skew symmetric bracket called
Courant bracket (see \cite{Liu, Bursztyn}). One can always choose
a splitting $s : TM \to \TT M$ with isotropic image (see
\cite{Severa, SeveraWeinstein}) such that one can identify the extended Courant
algebroid with the direct sum of the tangent and the cotangent
bundles of $M$, namely
\begin{eqnarray*}
\TT M & \stackrel{\cong}{\to}  & \TTB M:= TM \oplus T^*M \\
\gr{X} & \mapsto & (a(\gr{X}), \gr{X} - s(a(\gr{X}))) \\
\gr{X} & \mapsto & ( X + \xi). \end{eqnarray*}

With this identification the Courant bracket on $\TT M$ becomes
the $H$ twisted Courant bracket on $\TTB M$, where $H$ is a closed three
form determined by the splitting:
$$[ X + \xi, Y + \eta ]_H = [X,Y] + \LL_X \eta - \LL_Y \xi -
\frac{1}{2} d ( \iota_X \eta - \iota_Y \xi) - \iota_X \iota_Y H;$$
and the bilinear form becomes:
$$\langle X + \xi, Y + \eta \rangle = \frac{1}{2}(\iota_X \eta + \iota_Y \xi).$$

Let us emphasize that the splitting of the Courant algebroid is
not canonical; for any two form $B$, the B-field transformation of
the splitting
$$e^B(X + \xi ) =  X + \xi + \iota_X B $$
gives another splitting of $\TT M$ with twisting three form $H-dB$. The cohomology class $[H] \in H^3(M; \real)$ is called the \v Severa class of $\TT M$.

From now on we will work with a chosen splitting of the Courant
algebroid $\TT M$. Hence, the three form $H$ will also be fixed.

\subsection{Symmetries of the Exact Courant Algebroid}
The group
$ {\rm{Diff}} M \ltimes \Omega^2(M)$
with composition 
$$(\lambda, \alpha) \circ ( \mu, \beta) = (\lambda \mu, \mu^*
\alpha + \beta)$$ acts on $\TTB M$ in the following way
\begin{eqnarray*}(\lambda, \alpha) \circ (X + \xi) = \lambda_*X +
(\lambda^{-1})^*(\xi + \iota_X \alpha) = \lambda_* X +
(\lambda^{-1})^* \xi + \iota_{\lambda_*X} \alpha  \end{eqnarray*} and its induced action on the twisted Courant
bracket becomes
$$(\lambda, \alpha) \circ [X+\xi, Y + \eta]_H =
[(\lambda,\alpha)\circ(X+ \xi), (\lambda,\alpha) \circ (Y +
\eta)]_{(\lambda^{-1})^*(H-d\alpha)}.$$
Hence, the group of symmetries of the Exact Courant algebroid $\TTB M$ is
$$\GGC_H:= \{(\lambda, \alpha) \in {\rm{Diff}} M \ltimes \Omega^2(M) \ | \ H = (\lambda^{-1})^*(H-d\alpha) \}$$

The Lie algebra $\XXC_H$ of $\GGC_H$ is then
$$\XXC_H = \{(X,A) \in \Gamma(TM) \oplus \Omega^2(M) \ | \ dA = -\LL_XH=-d \iota_X H\}$$
with Lie bracket
$$[(X,A),( Y,B)] = ([X,Y], \LL_XB- \LL_YA ),$$
and with infinitesimal action on $\TTB M$ given by
$$(X,A)\circ (Y+\eta) = [X,Y]+\LL_X \eta - \iota_YA .$$
Note then that $(X,A)$ belongs to $\XXC_H$ if and only if $d(A + \iota_X H)=0$ (this equation will be of use later).

\begin{definition}
A Lie group $G$ acts on the exact Courant algebroid $\TTB M$ if
there is a homomorphism
\begin{eqnarray*}
G  \to  \GGC_H \ \ \ \ \ \ \ \ \ \ \ \ 
g  \mapsto  (\lambda_g, \alpha_g),
\end{eqnarray*}
and a Lie algebra $\gr{g}$ acts infinitesimally on $\TTB M$ if there is a Lie algebra homomorphism
\begin{eqnarray*}
\gr{g} \to  \XXC_H  \  \ \ \ \ \ \ \ \ 
a  \mapsto  (X_a, A_a).
\end{eqnarray*}
\end{definition}

\subsection{Extended symmetries}

Note that there is a homomorphism of algebras
$$\kappa: \Gamma(\TTB M) \to \XXC_H \ \ \ \ \ \ (X+\xi) \mapsto (X, d\xi - \iota_XH)$$
that sends the Courant bracket to the Lie bracket on $\XXC_H$. With this in mind we have an action of $\Gamma(\TTB M)$ on itself
given by the formula
$$(X+\xi) \circ (Y+\eta) = [X,Y] + \LL_X\eta - \iota_Y(d\xi - \iota_XH).$$

\begin{definition} \label{definition extended action}
We will say that the Lie Group $G$ (or the Lie algebra $\gr{g}$) acts by extended symmetries on $\TTB M$
 whenever: \begin{itemize} \item the infinitesimal action factors through
$\Gamma( \TTB M)$ as algebras, i.e. there is an algebra homomorphism 
\begin{eqnarray*}
\delta: (\gr{g}, [,])&  \to & (\Gamma(\TTB M), [,]_H)   \\
a & \mapsto & (X_a + \xi_a) 
\end{eqnarray*}
that makes the infinitesimal action be $a \mapsto (X_a, d\xi_a - \iota_{X_a}H)$, and
\item the image of $\gr{g}$ in $\Gamma(\TTB M)$ is an
isotropic subspace, in other words, for every $a, b \in
\gr{g}$
$$\langle X_a + \xi_a, X_b + \xi_b \rangle = 0,$$
and what is the same $$\iota_{X_a}\xi_b =- \iota_{X_b}
\xi_a \ \ \  \mbox{and} \ \ \ \ \iota_{X_a} \xi_a=0.$$
\end{itemize}
\end{definition}

\begin{rem}
We would like to point out that an ``extended action", as we have defined above and in \cite{HuUribe}, is equivalent to a ``lifted action" as is defined in section 2.3 of \cite{Bursztyn1} (we use the skew-symmetric version of the Courant bracket while in \cite{Bursztyn1} is used the non-skew-symmetric version, but as in both cases the image of the Lie algebra must be isotropic, then the two definitions agree). In \cite{Bursztyn, Bursztyn1} the authors call ``extended action" a more general type of construction that includes Courant algebras and Courant algebra morphisms ${\gr{a}} \to \Gamma(\TTB M)$ that extend the Lie algebra morphism $\gr{g} \to \Gamma(TM)$.
\end{rem}

\subsection{Extended equivariant cohomology}

In \cite{HuUribe} the last two authors have defined an equivariant cohomology for extended actions. Let us recall the construction.

Consider the complex of differential forms $\Omega^\bullet(M):=\Omega^{\rm{even}}(M) \oplus \Omega^{\rm{odd}}(M)$ but with $\integer_2$ grading given by the parity of the degree and odd differential $d_H := d - H \wedge$ where $H$ is a closed 3-form on $M$. The cohomology of this complex $H^\bullet(M,H)$ is known as the $H$-twisted cohomology of $M$.

For $\gr{X}=X+\xi \in \Gamma(\TTB M)$ consider the even operator $\LL_\gr{X}$ and the odd operator $\iota_\gr{X}$ that act on 
$\Omega^\bullet(M)$ in the following way: for $\rho \in \Omega^\bullet(M)$ we have
$$\iota_\gr{X} \rho = \iota_X \rho + \xi \wedge \rho  \ \ \ \mbox{and} \ \ \ \LL_\gr{X} \rho= \LL_X \rho +(d\xi - \iota_X H) \wedge \rho.$$

If we consider any two elements $\gr{X}, \gr{Y}$ that lie in the image of the map $\delta : \gr{g} \to \Gamma(\TTB M)$, then the operators
$\LL$, $\iota$ and $d_H$ behave in the following way with respect to the graded commutators (see \cite[Thm. 4.4.3]{HuUribe}):
\begin{eqnarray*}
[d_H, \iota_\gr{X} ] = \LL_\gr{X}  \ \ \ \ \  [\LL_\gr{X}, \LL_\gr{Y}] = \LL_{ [\gr{X},\gr{Y} ]_H} \ \ \ \ \ \  [\iota_\gr{X}, \iota_\gr{Y} ]=0 
\end{eqnarray*}
\begin{eqnarray*}
[\LL_\gr{X}, \iota_\gr{Y} ] =\iota_{[\gr{X},\gr{Y}]_H} \ \ \ \ \ [d_H, \LL_\gr{X}]=0 \ \ \ \mbox{and} \ \ [d_H, d_H]=0.
\end{eqnarray*}

Following Cartan \cite{Guillemin}, we will consider the algebra $\Omega^\bullet(M) \otimes \widehat{S}(\gr{g}^*)$ of formal series on $\gr{g}$ with values
in $\Omega^\bullet(M)$. The algebra $ \widehat{S}(\gr{g}^*)$ is the $\gr{a}$-adic completion of the symmetric algebra $S(\gr{g}^*)$ where $\gr{a}$ is the ideal generated by all elements without constant term (see \cite[Ch.10]{AtiyahMacdonald}); in the next sections it will be clarified why it is necessary to complete the symmetric algebra . If $\{a,b,c\dots \}$ is a base of $\gr{g}$
and $\Omega^a, \Omega^b, \dots$ are dual elements of even degree then $S(\gr{g}^*)=\real[\Omega^a, \Omega^b, \dots]$ and
$\widehat{S}(\gr{g}^*)=\real[[\Omega^a, \Omega^b, \dots]]$. Recall also that for any $a$ and $b$ in the Lie algebra, $\iota_a \Omega^b=0$, $d\Omega^b=0$ and therefore $\LL_a \Omega^b=f^b_{ca}\Omega^c$ where the structural constants of the Lie algebra $\gr{g}$ satisfy $[b,c]=f^a_{bc}a$.

If we consider an extended action $\delta: \gr{g} \to \Gamma(\TTB M)$ we can extend the action of the operators $\LL_{\delta(a)}$ and $\iota_{\delta(a)}$ on the generators $\Omega^\bullet(M) \otimes \widehat{S}(\gr{g}^*)$ in the natural way, namely
$$\LL_{\delta(a)}( \rho \otimes \Omega^b) := (\LL_{\delta(a)} \rho) \otimes \beta + \rho \otimes \LL_a \Omega^b \ \ \ \ \ \mbox{and} \ \ \ \ 
\iota_{\delta(a)} ( \rho \otimes \Omega^b) := (\iota_{\delta(a)} \rho) \otimes \Omega^b,$$
and we can define the extended equivariant differential
\begin{eqnarray*}
d_{\gr{g}, \delta} : \Omega^\bullet(M) \otimes \widehat{S}(\gr{g}^*)  \to  \Omega^\bullet(M) \otimes \widehat{S}(\gr{g}^*) \\
\end{eqnarray*}
as the odd operator $$d_{\gr{g}, \delta} := d_H \otimes 1 +  \Omega^a \iota_{\delta(a)}$$ where the sum sum goes over a base of $\gr{g}$ and we are
using the repeated index convention.

It is easy to check that
$$(d_{\gr{g}, \delta})^2 \rho = - \Omega^a \LL_{\delta(a)} \rho $$
and therefore the second two authors have proposed the following definition (see \cite[Def. 5.1.1]{HuUribe})

\begin{definition}
Let $\delta : \gr{g} \to \Gamma(\TTB M)$ be an extended action, then the $\gr{g}$-extended equivariant complex of $\TTB M$ is the $\integer_2$ graded complex
$$C_\gr{g}^\bullet(\TTB M;\delta) := \{\rho \in \Omega^\bullet(M) \otimes \widehat{S}(\gr{g}^*) \ | \ 
 \LL_{\delta(a)} \rho =0 \ \mbox{for all} \ a \in \gr{g}\};$$
with differential $d_{\gr{g},\delta}$. The cohomology of $H_\gr{g}^\bullet(\TTB M; \delta)$ of the complex $C_\gr{g}^\bullet(\TTB M; \delta)$
is the extended $\gr{g}$-equivariant De Rham cohomology of $\TTB M$ under the extended action defined by $\delta$.
\end{definition}


Let us note that the extended $\gr{g}$-equivariant cohomology does not depend on the choice of splitting for $\TT M$ and its isomorphism class depends only on the \v Severa class of the Courant algebroid. If one performs a $B$-field transform, the action transforms $\delta$ to to $\delta'(a)=X_a+\xi_a + \iota_{X_a}B$, one gets the isomorphism $e^B : \Omega^\bullet(M) \to \Omega^\bullet(M)$ and therefore one obtains a quasi-isomorphism of complexes
$C_\gr{g}^\bullet(\TTB M;\delta) \to C_\gr{g}^\bullet(\TTB M;\delta')$ with $d_{\gr{g}, \delta'} := d_{H-dB} \otimes 1 +  \Omega^a \iota_{\delta'(a)}$ that induces an isomorphism of cohomologies $H_\gr{g}^\bullet(\TTB M; \delta) \cong H_\gr{g}^\bullet(\TTB M; \delta')$.

Also, if we take the Cartan complex
$$C_\gr{g}^*(M) =\{ \rho \in \Omega^*(M) \otimes {S}(\gr{g}^*) \ | \ 
\LL_{X_a} \rho=0  \ \mbox{for all} \ a \in \gr{g}\}$$
 with differential $d_\gr{g} \rho = d \rho - \Omega^a\iota_{X_a} \rho$, then the extended $\gr{g}$-equivariant complex  $C_\gr{g}^\bullet(\TTB M;\delta)$ becomes a module over $C_\gr{g}^*(M)$ and therefore the extended $\gr{g}$ equivariant cohomology $H_\gr{g}^\bullet(\TTB M;\delta)$ is a module over the equivariant cohomology $H_\gr{g}^*(M)$. In particular we have that $H_\gr{g}^\bullet(\TTB M;\delta)$ is also a module over $S(\gr{g}^*)^\gr{g}=H_\gr{g}^*( \cdot)$.

So far we do not know whether the extended equivariant cohomology fulfills all the properties of an equivariant cohomology theory, nor if it has a topological counterpart. Nevertheless, in the case that the group $G$ is compact, the extended equivariant cohomology turns out to be 
equivalent to what is known as twisted equivariant cohomology. This will be the subject of the next sections.

\section{Twisted equivariant cohomology}\label{twisted}

In as much as the twisted cohomology is defined twisting the differential of the De Rham complex with a closed 3-form, the twisted
equivariant cohomology is defined by twisting the equivariant differential with a closed and equivariant 3-form. 

We will define the twisted equivariant cohomology using the Cartan model and we will explicitly show the relation with the twisted Weil model. Then we will generalize the Chern-Weil map for twisted equivariant cohomology and we will finish by giving the topological counterpart of the twisted equivariant cohomology.
 \begin{rem}
 We would like to emphasize the fact that the twisted cohomology is a $\integer_2$-graded theory; therefore all inverse limits that will be carried out in this section will be $\integer_2$ -graded. Let us see the difference with a simple example:
\begin{quotation}
Take the $\integer$-graded rings $H^*(\complex P^n)=\real[x]/x^{n+1}$ where $|x|=2$. The inverse limit of these $\integer$-graded rings is the polynomial algebra $H^*(\complex P^\infty)=\real[x]$. Now, if one consider the same rings $H^\bullet(\complex P^n)=\real[x]/x^{n+1}$ but $\integer_2$-graded, as in the case of twisted cohomology, the inverse limit of these rings gives the algebra of formal series
  $H^\bullet(\complex P^\infty)=\real[[x]]$.
\end{quotation}
  To distinguish between $\integer$ and $\integer_2$ graded theories we will denote the former with an asterisk ${H}^*$ and the latter with a bullet ${H}^\bullet$.
  
  \end{rem}
Let us start by recalling the models of Cartan and Weil for the equivariant coholomogy (see \cite{MathaiQuillen}).

\subsection{Equivariant cohomology}

Following Weil one introduces a universal model for the curvature
and connection on a principal $G$ bundle.  The Weil algebra is
then by definition
$$W(\gr{g}) := S(\gr{g}^*) \otimes \Lambda(\gr{g}^*)$$
the tensor product of the symmetric algebra and the exterior
algebra of $\gr{g}^*$. If we denote with lower case letters
$a,b,c, \dots$ a base for the Lie algebra $\gr{g}$ then we will
denote by $\theta^a$ the variables dual to $a$ of degree one that
generate the exterior algebra, and by $\Omega^a$ the variables of
degree two that generate the symmetric algebra.

The derivations and contractions on this algebra are generated by
$$ \iota_a \theta^b = \delta_{ab} \ \ \ \ \ \ \iota_a \Omega^b= 0
\ \ \ \ \ d\theta^a=\Omega^a- \frac{1}{2}f^a_{bc} \theta^b\theta^c
\ \ \ \ \ d\Omega^a= f^a_{bc} \Omega^b \theta^c$$ where $f^a_{bc}$
are the structural constants: $[b,c]= f^a_{bc} a$.

The tensor product $ \Omega^*(M) \otimes W(\gr{g}) $ is a
differential graded algebra with a $\gr{g}$ action and derivations
$\iota_a$ satisfying the standard identities of the contractions.
The contractions $\iota_a$ act on the differential forms
$\Omega^*(M)$ by contracting on the direction of the vector field
$X_a$ that $a$ generates, but to make the notation less heavy we
will simply denote $\iota_a$ the operator $\iota_{X_a}$. The Lie derivative $\LL_a$ is 
defined by the Cartan formula $\LL_a = d\iota_a + \iota_a d$.

 The basic
subalgebra
$$\Omega_\gr{g}^*(M) := \{ \Omega^*(M) \otimes W(\gr{g}) \}_{bas} := \bigcap_{a} \left(\ker \LL_a \cap \ker \iota_a\right)$$ 
is a
differential graded algebra whose elements are called (Weil)
equivariant differential forms and whose cohomology $H^*(\Omega^*_\gr{g}(M))$ is the $G$-equivariant
cohomology of $M$.

Cartan \cite{Cartan} cf. \cite{MathaiQuillen} showed that there is
a smaller model for the equivariant forms which is given by the
$G$-invariant forms on $(\Omega^*(M) \otimes S(\gr{g}^*) )^\gr{g}$ with
differential given by $d_\gr{g} = d - \Omega^a \iota_a.$ Let us denote the cohomology
of this complex  by $$H^*_{\gr{g}}(M) := H^*\left((\Omega^*(M) \otimes S(\gr{g}^*)
)^\gr{g}, d_\gr{g}\right).$$

Cartan showed that there is a homomorphism \begin{eqnarray*}j:
\Omega^*(M) \otimes S(\gr{g}^*) 
 & \to &   \Omega^*(M) \otimes W(\gr{g})\\
\alpha & \mapsto & \prod_a (1 - \theta^a \iota_a) \alpha 
\end{eqnarray*}
that induces a quasi-isomorphism of complexes
\begin{equation}
j : (\Omega^*(M) \otimes S(\gr{g}^*) )^\gr{g} \stackrel{\sim}{\to} \Omega^*_\gr{g}(M)  \label{quasiiso j}
\end{equation} and therefore the
map $j$ induces an isomorphism of equivariant cohomologies
$$j: H^*_\gr{g}(M) \stackrel{\cong}{\to} H^*(\Omega^*_\gr{g}(M)).$$

\subsection{Twisted equivariant cohomology}

Let us start by taking a closed equivariant 3-form $\HH= H + \Omega^a\xi_a$ on the Cartan complex $(\Omega^*(M)\otimes S(\gr{g}^*))^\gr{g}$
with $H$ a 3-form and $\xi_a$ 1-forms on $M$.

The fact that $\HH$ is closed implies that
$$d_\gr{g} \HH = (d-\Omega^b\iota_b)(H + \Omega^a \xi_a) = dH + \Omega^a
(d\xi_a - \iota_a H) - \Omega^a \Omega^b \iota_b \xi_a=0 $$
which happens if and only if
$$dH=0 \ \ \ \ d\xi_a - \iota_a H=0 \ \ \ \mbox{and } \ \ \ \iota_a \xi_b =-\iota_b \xi_a.$$

The fact that $\HH$ is equivariant implies that for all $b \in \gr{g} $
$$\LL_b\HH = \LL_b \HH + (\LL_b \Omega^a) \xi_a +
\Omega^a
(\LL_b \xi_a) = \LL_b \HH +  f^a_{cb} \Omega^c \xi_a + \Omega^a(\LL_b \xi_a)=0;$$
this happens if and only if $\HH$ is $\gr{g}$ invariant and $\xi_{[b,a]}= \LL_b\xi_a$.

\begin{rem}
For an extended action $\delta : \gr{g} \to \Gamma(\TTB M)$ the three form $H + \Omega^a\xi_a$, with $H$ is the twisting form and $\xi_a= \delta(a)-X_a $, would be closed and equivariant form if and only if $d\xi_a = - \iota_aH$, as this would imply that $\LL_aH=d\iota_aH=-dd\xi_a=0$. 
\end{rem}

Let us now see what is the image in the Weil model of
the three form $\HH$. We need this in order to have a very
explicit description of the twisted Chern-Weil homomorphism.
Fortunately the expression of the three form turned out to be very
simple, its proof not; we will reproduce here the calculations.

\begin{proposition}
The image of $\HH$ in $\Omega^*_{\gr{g}}(M)$ under the quasi-isomorphism
$j$ defined in (\ref{quasiiso j}) is the basic three form
$${\bf H} := j(\HH)=H + d(\theta^a \xi_a) -\frac{1}{2}d(\theta^p\theta^q
\iota_q \xi_p).$$
\end{proposition}

\begin{proof}
We will proceed by expanding the derivations in the three form
${\bf H}$ and we will compare them with the expansion of $j(\HH)$.

Let us start by expanding ${\bf H}$:
\begin{eqnarray}
{\bf H} & = & H+  d(\theta^a \xi_a) -\frac{1}{2}d(\theta^p\theta^q
\iota_q \xi_p) \nonumber \\
& = & H + \Omega^a \xi_a -\frac{1}{2}f^a_{bc}
\theta^b\theta^c\xi_a -\theta^a d\xi_a -\frac{1}{2}\Omega^p
\theta^q \iota_q \xi_p +
\frac{1}{4}f^p_{rs}\theta^r\theta^s\theta^q\iota_q \xi_p \nonumber\\
& & + \frac{1}{2}\theta^p \Omega^q \iota_q \xi_p -
\frac{1}{4}f^q_{tu}\theta^p\theta^t\theta^u\iota_q\xi_p -
\frac{1}{2}\theta^p\theta^q d\iota_b\xi_p \nonumber \\
 &=& H+\Omega^a \xi_a
-\frac{1}{2} \theta^b\theta^c\xi_{[b,c]} -\theta^a d\xi_a -
\Omega^p \theta^q \iota_q \xi_p +
\frac{1}{2}\theta^r\theta^s\theta^q\iota_q \xi_{[r,s]} \nonumber\\
& &  - \frac{1}{2}\theta^p\theta^q d\iota_b\xi_p. \label{Weil
three form}
\end{eqnarray}

Numerate from left to right the expressions in line (\ref{Weil
three form}). We will match them with the expansion of $j(H_G)$.
Let us calculate then:

\begin{eqnarray}
j(H_G) & = & (1-\theta^e\iota_e)(1-\theta^c\iota_c)(1 -
\theta^b\iota_b)(H + \Omega^a\xi_a) \nonumber\\
& = & H + \Omega^a\xi_a - \theta^b\iota_b(H + \Omega^a\xi_a)
+\frac{1}{2} \theta^b\theta^c\iota_c\iota_b(H+ \Omega^a\xi_a)
\nonumber \\
&& -\frac{1}{6}\theta^b\theta^c\theta^e\iota_e\iota_c\iota_b(H+
\Omega^a\xi_a)\nonumber \\
&=& H+ \Omega^a\xi_a -\theta^bd\xi_b -\theta^b
\Omega^a\iota_b\xi_a + \frac{1}{2}\theta^b\theta^c \iota_c d\xi_b
-\frac{1}{6}\theta^b\theta^c\theta^e\iota_e\iota_c d \xi_b
\nonumber
\end{eqnarray}
As we have that
$$\iota_c d\xi_b = \LL_c \xi_b - d \iota_c \xi_b= \xi_{[c,b]} - d \iota_c
\xi_b,$$ then
\begin{eqnarray}
j(H_G) &=& H+ \Omega^a\xi_a -\theta^bd\xi_b -\theta^b
\Omega^a\iota_b\xi_a + \frac{1}{2}\theta^b\theta^c \xi_{[c,b]}
 - \frac{1}{2}\theta^b\theta^c d \iota_c \xi_b \nonumber\\
& & -\frac{1}{6}\theta^b\theta^c\theta^e\iota_e\iota_c d \xi_b.
\label{expansion j(HG)}
\end{eqnarray}
We see that the first 6 terms in (\ref{expansion j(HG)}) match all
but the sixth term in (\ref{Weil three form}). Let us then expand
the last term in (\ref{expansion j(HG)}):
\begin{eqnarray*}
-\frac{1}{6}\theta^b\theta^c\theta^e\iota_e\iota_c d \xi_b &=&
-\frac{1}{6}\theta^b\theta^c\theta^e\iota_e(\LL_c\xi_b -
d\iota_c\xi_b)\\
&=& -\frac{1}{6}\theta^b\theta^c\theta^e \left( \iota_e
\xi_{[c,b]} - \LL_e \iota_c \xi_b \right)\\
&=&-\frac{1}{6}\theta^b\theta^c\theta^e \left( \iota_e \xi_{[c,b]}
-\iota_c \LL_e \xi_b - \iota_{[e,c]} \xi_b \right)\\
&=&-\frac{1}{6}\theta^b\theta^c\theta^e \left( \iota_e \xi_{[c,b]}
-\iota_c \xi_{[e,b]} + \iota_{b} \xi_{[e,c]} \right)\\
&=&-\frac{1}{2}\theta^b\theta^c\theta^e  \iota_e \xi_{[c,b]}
\end{eqnarray*}
and we can see that it matches the sixth term in (\ref{Weil three
form}). Here we have used the fact that
$$\iota_{[e,c]} \xi_b = f^a_{ec} \iota_a\xi_b =- f^a_{ec} \iota_b
\xi_a= - \iota_b \xi_{[e,c]}.$$
\end{proof}

Let us note that as ${\bf H}$ is basic we have that $\iota_a
{\bf H}=0$, $\LL_a {\bf H}=0$ and  $d{\bf H}=0$.

\vspace{0.5cm}

Having in hand the closed three forms defined previously,
 we can now change the differential in the Cartan model as
well as in the Weil model. But as the twisted equivariant cohomology is a $\integer_2$-graded
theory, we first need to complete the symmetric algebra $S(\gr{g}^*)$ in both the Cartan and the Weil model.
Then let us denote $\widehat{W}(\gr{g}) = \widehat{S}(\gr{g}^*)\otimes \Lambda(\gr{g}^*)$. 

Now, in the Cartan model define the twisted equivariant differential as
$$d_{\gr{g},\HH} := d_\gr{g} - \HH\wedge$$
and in the Weil model as
$$d_{\bf H}:= d - {\bf H}\wedge.$$
As we have that $(d_{\gr{g},\HH})^2=0$ and $(d_{\bf H})^2=0$, we can define:
\begin{definition}
The (Cartan) twisted equivariant cohomology is the cohomology of
the complex of $\gr{g}$-invariant forms of $\Omega^\bullet(M) \otimes \widehat{S}(\gr{g}^*)$ and  the twisted differential $d_{\gr{g},\HH}$,
i.e.
$$H^\bullet_\gr{g}(M,\HH):= H^\bullet((\Omega^\bullet(M) \otimes \widehat{S}(\gr{g}^*))^\gr{g}; d_{\gr{g},\HH}).$$
The (Weil) twisted equivariant cohomology is the cohomology of the
basic forms $\widehat{\Omega}^\bullet_\gr{g}(M)= ( \Omega^\bullet(M)\otimes \widehat{W}(\gr{g}))_{bas}$
with the twisted differential $d_{\bf H}$, i.e.
$$H^\bullet(\widehat{\Omega}^\bullet_\gr{g}(M), {\bf H}): = H^\bullet( \widehat{\Omega}^\bullet_\gr{g}(M); d_{\bf H}).$$
\end{definition}

As the map $j : \Omega^\bullet(M) \otimes \widehat{S}(\gr{g}^*) \to \Omega^\bullet (M) \otimes \widehat{W}(\gr{g})$
induces a quasi-isomorphism of complexes $j : (\Omega^\bullet(M) \otimes \widehat{S}(\gr{g}^*))^\gr{g} \to \widehat{\Omega}^\bullet_\gr{g} (M)$
 and $j(H_G)={\bf H}$, then we can
conclude that $j$ also induces a quasi-isomorphism of twisted
complexes. So we have
\begin{proposition}
The Cartan and the Weil twisted equivariant cohomologies are
isomorphic, $$j:H^\bullet_\gr{g}(M,\HH) \stackrel{\cong}{\to}
H^\bullet(\widehat{\Omega}^\bullet_\gr{g}(M), {\bf H}).$$
\end{proposition}

In \cite{HuUribe} the last two authors have shown that the twisted equivariant
cohomology possesses all the properties of a cohomology theory.
Moreover, as the twisted equivariant cohomology is a module over
the symmetric algebra $S(\gr{g}^*)^\gr{g}$ (the equivariant cohomology of
a point) then the last two authors have shown that a generalization of the
localization theorem of Atiyah-Bott \cite{AtiyahBott} in the case of torus actions holds for
the twisted equivariant cohomology; namely, for $F=M^G$ the fixed
point set of the action, and $i:F \to M$ the inclusion, one has
that in a suitable localization, for all classes $x \in
H^\bullet_\gr{g}(M,\HH)$ the following formula holds
$$x = \sum_{Z \subset F} i^Z_*(i_Z^*(x)) \wedge e_G(\nu_Z)^{-1}$$
where $Z$ runs over the connected components of $F$, and
$e_G(\nu_Z)$ is the equivariant Euler class of $\nu_Z$ the normal
bundle of $Z$ in $M$ (see Theorem 7.2.5 in \cite{HuUribe}).

\section{Extended equivariant cohomology for compact Lie groups}\label{compact}

In this section we will show that in the case of an extended action of a compact Lie group, one can find an appropriate choice
of splitting of the exact Courant algebroid in such a way that the extended equivariant cohomology becomes the twisted equivariant cohomology
of the previous section. 

Let us first show how the group of symmetries $\GGC_H$ change after performing a $B$-field transform, thus changing the splitting of $\TTB M$.
For $(\lambda, \alpha) \in \GGC_H$ it is easy to check that
$$(\lambda, \alpha + \lambda^*B - B) \circ e^B(X+\xi) = e^B((\lambda,\alpha) \circ (X+\xi)),$$
and therefore the group of symmetries changes to
\begin{eqnarray}
e^B : \GGC_H & \to & \GGC_{H-dB} \nonumber \\
(\lambda, \alpha) & \mapsto & (\lambda, \alpha +\lambda^*B - B). \label{B-field group}
\end{eqnarray}
Let us denote $\overline{H}:=H-dB$.

\begin{lemma} [Prop. 2.11 \cite{Bursztyn}] \label{lemma invariant splitting}
Consider the action of a compact Lie group $G$ on the extended Courant algebroid $\TTB M$ given by the map $G \to \GGC_H$, $g \mapsto (\lambda_g, \alpha_g)$. Then there exists a 2-form $B$ such that $$\alpha_g + \lambda_g^*B - B=0 \ \ \ \mbox{for all} \ g \in G; $$
and therefore after changing the splitting with $B$ (see equation \ref{B-field group}), the action of $G$ is only given by diffeomorphisms,
i.e. $G \to \GGC_{\overline{H}}$, $g \mapsto (\lambda_g, 0)$. Moreover, the 3-form $\overline{H}=H-dB$ is $G$ invariant.
\end{lemma}

\begin{proof}
For  $g \in G$ and $B \in \Omega^2(M)$ let us define the 2-form $$g \cdot B := (\lambda_g^{-1})^*( B-\alpha_g ).$$
This becomes an action of $G$ on $\Omega^2(M)$ as we have that
\begin{eqnarray*}
(hg)\cdot B&=& ((\lambda_{hg}^{-1})^*(B- \alpha_{hg} ) = (\lambda_{h}^{-1})^*(\lambda_{g}^{-1})^*(B-\lambda_{g}^*\alpha_h - \alpha_g )\\
& =& (\lambda_{h}^{-1})^* \left[  (\lambda_{g}^{-1})^*(B-\alpha_g ) - \alpha_h \right] = h\cdot(g \cdot B).
\end{eqnarray*}
Now, taking a $G$-invariant metric $d\mu$ with total volume 1, we can define the 2-form
$$B:= \int_G (h \cdot 0) \ d\mu(h) = \int_G (\lambda_h^{-1})^*\alpha_h \ d\mu(h)$$
which clearly satisfies $g \cdot B =B$ for all $g \in G$, and therefore we have that 
$\lambda_g^*B = B -\alpha_g$.

The fact the $\overline{H}$ is invariant follows from the definition of $\GGC_{\overline{H}}$; for $({\lambda}, \overline{\alpha}) \in \GGC_{\overline H}$ we have that $\lambda^*\overline{H}= H - d\overline{\alpha}$. But for all $g \in G$ we have that $\overline{\alpha}_g=0$, then $\lambda_g^*\overline{H}=\overline{H}$.
\end{proof}

In the case that the compact Lie group $G$ acts by extended symmetries $$\delta: \gr{g} \to \Gamma(\TTB M) \to \XXC_H$$
 $$a \mapsto (X_a,\xi_a) \mapsto (d\xi_a - \iota_{X_a}H), $$ the change of splitting of lemma \ref{lemma invariant splitting} defines
 a map $\overline{\delta} : \gr{g} \to \Gamma(\TTB M)$, $$ \overline{\delta}(a) = X_a + \overline{\xi}_a = X_a + \xi_a + \iota_{X_a}B,$$
and as the action is only given by diffeomorphisms we have that the 2-forms $d\overline{\xi}_a - \iota_{X_a} \overline{H}$ are all equal to zero.
Therefore we can conclude,
\begin{cor} \label{cor d xi - iota H}
If the compact Lie group acts by extended symmetries and we perform the $B$-field transform of lemma \ref{lemma invariant splitting}, then for all
$a \in \gr{g}$ we have that $$d\overline{\xi}_a - \iota_{X_a} \overline{H} =0$$ where the infinitesimal action is given by $\overline{\delta}(a)=X_a+\overline{\xi}_a$.
\end{cor}

\begin{lemma} [Thm. 2.13 \cite{Bursztyn}] \label{overline HH}
Let the compact Lie group $G$ act on $\TTB M$ by extended symmetries. Then, after performing 
the $B$-field transform of lemma \ref{lemma invariant splitting}, the 3-form
$$ \overline{\HH} := \overline{H} + \Omega^a \overline{\xi}_a \in \Omega^*(M) \otimes S(\gr{g}^*)$$
 becomes $G$-invariant and $d_G$-closed. Hence, it defines an equivariant De Rham class $[\overline{\HH}] \in H^3_G(M)$.
\end{lemma}

\begin{proof}
Let us first check that $\overline{\HH}$ is $G$-invariant, so for $a,b,c\dots$ a base of $\gr{g}$ we have that
\begin{eqnarray*}
\LL_b \overline{\HH}& =& \LL_b \overline{H} + (\LL_b \Omega^a) \overline{\xi}_a + \Omega^a (\LL_b \overline{\xi}_a)\\
&= &\LL_b \overline{H} + (f^a_{cb} \Omega^c) \overline{\xi}_a + \Omega^a (\overline{\xi}_{[b,a]})\\
&=& 
 \LL_b \overline{H} + \Omega^c \overline{\xi}_{[c,b]} + \Omega^a \overline{\xi}_{[b,a]} \\ &=& 0
 \end{eqnarray*}
 where we have used that $\LL_b \overline{\xi}_a = \overline{\xi}_{[b,a]}$. This last equality  
 follows from the fact that 
  the map $\gr{g} \to \Gamma(\TTB M)$ is an algebra map, and so we have
  \begin{eqnarray*}\overline{\xi}_{[a,b]}& =& \LL_{X_a}\overline{\xi}_b - \LL_{X_b}
\overline{\xi}_a -\frac{1}{2}d(\iota_{X_a}\overline{\xi}_b - \iota_{X_b}
\overline{\xi}_a) + \iota_{X_b}\iota_{X_a}\overline{H}\\
& = & \LL_{X_a} \overline{\xi}_b - d\iota_{X_b}
 \overline{\xi}_a -\iota_{X_b}d\overline{\xi}_a +d \iota_{X_b}\overline{\xi}_a + \iota_{X_b}d\overline{\xi}_a\\
& = &\LL_{X_a}\overline{\xi}_b. \end{eqnarray*}

 Now let us calculate $d_G \overline{\HH}$:
 \begin{eqnarray*}
 d_G \overline{\HH} = dH + \Omega^a (d\overline{\xi}_a - \iota_a \overline{H}) + \Omega^b\Omega^c(\iota_b \overline{\xi}_a + \iota_a \overline{\xi}_b)
 \end{eqnarray*}
 and as $H$ is closed, $d\overline{\xi}_a - \iota_a \overline{H}=0$ becasue of lemma \ref{cor d xi - iota H} and 
 $$\iota_b \overline{\xi}_a + \iota_a \overline{\xi}_b = \langle X_a + \overline{\xi}_a, X_b+ \overline{\xi_b} \rangle =0$$ as the action is isotropic (see
  definition \ref{definition extended action}), then $d_G \overline{\HH}=0$.
\end{proof}

Knowing that the 3-form $\overline{\HH}$ is invariant and closed, we can conclude this section with the following result

\begin{theorem} \label{isomorphism to overline HH}
Let the compact Lie group $G$ act on $\TTB M$ by extended symmetries. Then the extended equivariant cohomology $H^\bullet_\gr{g}(\TTB M, \delta)$ is isomorphic to the twisted equivariant cohomology
$H^\bullet_G(M, \overline{\HH})$ where $\overline{\HH}$ is the equivariant closed 3-form of lemma \ref{overline HH}.
\end{theorem}

\begin{proof}
Let us perform the $B$-field transform of lemma \ref{lemma invariant splitting}. Then we have  a quasi-isomorphism of extended complexes
$C^\bullet_\gr{g}(\TTB M, \delta) \to C^\bullet_\gr{g}(\TTB M, \overline{\delta})$ that induces an isomorphism of cohomologies
$$H^\bullet_\gr{g}(\TTB M, \delta) \cong H^\bullet_\gr{g}(\TTB M, \overline{\delta}).$$

Now for $\rho \in \Omega^\bullet(M)$, we have that $$\LL_{\overline{\delta}(a)} \rho = \LL_a \rho + (d \overline{\xi}_a - \iota_a \overline{H}) \wedge \rho = \LL_a \rho,$$
and therefore the extended complex becomes
$$C_\gr{g}^\bullet(\TTB M, \overline{\delta}) = \{ \rho \in \Omega^\bullet(M) \otimes \widehat{S}(\gr{g}^*)  \  | \ \LL_a \rho =0 \ \mbox{for all} \ a \in \gr{g} \}$$
with derivative $d_{\gr{g}, \overline{\delta}} := d_{\overline{H}} + \Omega^a \iota_{\overline{\delta}(a)}$ which can be expanded and transformed into
\begin{eqnarray*}
d_{\gr{g}, \overline{\delta}} & = & d - \overline{H}\wedge + \Omega^a \iota_a + \Omega^b \overline{\xi}_b \wedge \\
& = & d +  \Omega^a \iota_a - ( \overline{H} - \Omega^b \overline{\xi}_b)\wedge \\
 & = & d_\gr{g} - \HH \wedge\\ & =& d_{\gr{g}, \overline{\HH}}.
\end{eqnarray*}

Thus the extended complex $C_\gr{g}^\bullet(\TTB M, \overline{\delta})$ is the same as the twisted complex
$$C_G^\bullet(M,\overline{\HH}) =(\Omega^\bullet(M) \otimes \widehat{S}(\gr{g}^*))^G ; d_{\gr{g}, \overline{\HH}}).$$
Therefore the cohomologies $H^\bullet_\gr{g}(\TTB M, \overline{\delta}) $ and $H^\bullet_G(M, \overline{\HH})$ are the same. The theorem follows.
\end{proof}


Let us finish this section with an example.

\begin{example} \label{Example circle circle}
Let $G=S^1$ and $M=S^1$ where $G$ acts trivially on $M$, but with extended action given by the map
$$\delta : \real \to \Gamma(\TTB S^1), \ \ \  \delta(1) = d\theta$$
where $d\theta \in \Omega^1(S^1)$.
The extended equivariant cohomology $H^\bullet_\gr{g}(\TTB M, \delta)$ becomes the cohomology of the complex
$$\Omega^\bullet(S^1) \otimes \real[[ \Omega]] \ \ \mbox{with differential} \ \ d_{\gr{g}, \delta}=d - d\theta \Omega \wedge$$ 
which is the twisted equivariant cohomology $H_{S^1}^\bullet(M,d\theta \Omega)$.

One can check easily that the closed forms are the odd forms. Now let us see which odd forms are exact. Consider the odd form $f_i d\theta  \Omega^i$ and the even form $g_j \Omega^j$. The equation $d_{\gr{g}, \delta} (g_j \Omega^j)=f_i d\theta  \Omega^i$ is equivalent to the equations
$dg_0 = f_0 d\theta$ and $dg_i - g_{i-1}d\theta = f_i d\theta$ for $i>0$. If $\int_{S^1} f_0 d\theta =0$ these equations are solved inductively starting from $0$ and making sure that one chooses the $g_i$ such that their integral satisfies $\int_{S^1} (g_{i} + f_{i+1})d \theta =0$.

Then the twisted equivariant cohomology is
$$H_{S^1}^{1}(M,d\theta \Omega) =\real   \ \ \ \mbox{and} \ \ \ H_{S^1}^{0}(M,d\theta \Omega)=0.$$

\end{example}

With this particular example one can show what happens in the case that the twisted cohomology is defined without completing the symmetric algebra. The calculations will be done in the Appendix, section \ref{appendix}.

\section{Twisted Chern-Weil homomorphism} \label{chern}

In this section we will extend the Chern-Weil homomorphism for the twisted case.  Then let us start by recalling the basics
of Chern-Weil theory. From now on the Lie group $G$ will be compact and therefore all equivariant cohomologies will have the subscript $G$.

Let $P$ be a principal $G$ bundle together with its connection and curvature
$$\theta \in (\Omega^1(P) \otimes \gr{g})^G \ \ \ \ \ \ \Omega  \in (\Omega^2(P) \otimes \gr{g})^G,$$
satisfying the identities
$$\iota_X \theta = X, \ \ \  \iota_X \Omega=0, \ (X \in \gr{g}) \ \ \ \Omega = d\theta + \frac{1}{2}[\theta,\theta] \ \mbox{and} \ \ d\Omega= [\Omega,\theta].$$
The connection and curvature determine maps $$\gr{g}^* \to \Omega^1(P), \ \ \ \gr{g}^* \to \Omega^2(P)$$
that induce a homomorphism of graded algebras
$$W(\gr{g}) \to \Omega^*(P)$$ which is the unique homomorphism carrying the universal connection and curvature to the connection
and curvature of $P$. This homomorphism is called the Weil homomorphism.

If $M$ is a manifold with a $G$-action, then the Weil homomorphism for the $G$-principal bundle
$M \times P \to M \times_G P$ (where $G$-acts diagonally) combined with the lifting of forms from $M$
to $P \times M$ determine a homomorphism
$$w : \Omega^*(M) \otimes W(\gr{g})  \to \Omega^*(M \times P)$$
wich induces a map of basics subalgebras
$$\overline{w} : \Omega^*_G(M) \to \Omega^*(M \times P)_{bas} \cong \Omega^*(M \times_G P)$$
which is a homomorphism of differential graded algebras. This map is known as the Chern-Weil homomorphism
determined by the connection in $M \times P$.

The induced homomorphism in cohomologies
$$\overline{w} : H^*(\Omega^*_G(M)) \to H^*(M \times_G P)$$ is independent of the connection in $P$.
Following \cite[Th. 2.5.1, Pr. 2.5.5]{Guillemin}, we can choose finite dimensional manifolds $EG_k$ with free $G$ actions, and equivariant inclusions $EG_k \to EG_{k+1}$ in such a way that $EG=\lim_\to EG_k$ becomes a model for the universal $G$ principal bundle with $EG$ contractible.
Then we have that $\Omega^*(EG) = \lim_{\leftarrow} \Omega^*(EG_k)$ and that this complex is acyclic. Moreover, the 
Chern-Weil map for each $k$ 
$$\overline{w}_k : \Omega^*_G(M) \to \Omega^*(M \times_G EG_k)$$
induces a map
\begin{eqnarray}
\label{Chern-Weil map Borel} \overline{w} : \Omega^*_G(M) \to \Omega^*(M \times_G EG)=\lim_\leftarrow \Omega^*(M \times_G EG_k)
\end{eqnarray}
which becomes a quasi-isomorphism of complexes and therefore an isomorphism of cohomologies
$$H^*(\Omega^*_G(M)) \cong H^*(M\times_G EG).$$

Let us show that this result can be generalized to the twisted case.

\begin{theorem} \label{Theorem Twisted Chern-Weil}
Let $\bf{H}$ be a closed equivariant 3-form in $\Omega^3_G(M)$ and consider $H_k:=\overline{w}_k(\bf{H})$ and $H =\lim_\leftarrow H_k$.
Then the twisted Chern-Weil homomorphisms
$$\phi_k : (\widehat{\Omega}^\bullet_G(M), d_{\bf{H}}) \to (\Omega^\bullet(M\times_G EG_k), d_{H_k})$$
induce a homomorphism
$$\phi : (\widehat{\Omega}^\bullet_G(M), d_{\bf{H}}) \to (\Omega^\bullet(M\times_G EG),d_H) :=\lim_{\leftarrow k}(\Omega^\bullet(M\times_G EG_k), d_{H_k})$$
which induces an isomorphism in twisted cohomologies
$$\phi : H^\bullet(\widehat{\Omega}^\bullet_G(M), {\bf{H}}) \cong H^\bullet(M\times_G EG, H):=H^\bullet(\Omega^\bullet(M\times_G EG),d_H).$$
\end{theorem}

\begin{proof}
Even though the complexes $\widehat{\Omega}^\bullet_G(M)$ and $\Omega^\bullet(M\times_G EG)$ are $\integer_2$
graded, we could use the $\integer$-grading of their untwisted versions to define the filtration $F^p\widehat{\Omega}^\bullet_G(M)$
and $F^p\Omega^\bullet(M\times_G EG)$ of all forms of degree greater or equal than $p$. 
The map $\phi$ becomes a homomorphism of filtered complexes and therefore it gives rise to a map of spectral sequences with $p$-th term
$\phi_p: E_p^{*,*}\to \overline{E}_p^{*,*}$.

The first terms of these spectral sequences are $$E_1^{*,*}= (\Omega^*_G(M),d) \ \ \ \mbox{and} \ \ \ \ \overline{E}_1^{*,*}=(\Omega^*(M\times_G EG),d)$$ and $\phi_1$ is simply the Chern-Weil map $\overline{w}$  of (\ref{Chern-Weil map Borel}). By the equivariant de Rham theorem the second terms
become isomorphic
$$\phi_2: E_2^{*,*}= H^*(\Omega^*_G(M)) \stackrel{\cong}{\to} \overline{E}_2^{*,*}= H^*(M\times_G EG),$$
then by Theorem 3.9 of \cite{McCleary} if the filtrations are exhaustive and complete we would have that $\phi$ induce an isomorphism
of twisted cohomologies
$$\phi: H^\bullet(\widehat{\Omega}^\bullet_G(M), {\bf H}) \cong H^\bullet(M\times_G EG , H).$$

Let us finish the proof by showing that both filtrations are complete. The filtrations are exhaustive because the filtrations were defined by the degree.

The twisted cohomology $H^\bullet(\widehat{\Omega}^\bullet_G(M), {\bf H})$ is complete because the twisted complex is complete; this follows from the following equalities
$$\Omega(M)\otimes \widehat{W}(\gr{g}) = \lim_\leftarrow \Omega(M)\otimes \widehat{W}(\gr{g})/F^p\widehat{S}(\gr{g}^*)=\lim_\leftarrow \Omega(M)\otimes \widehat{W}(\gr{g})/(F^p\Omega(M)\otimes\widehat{W}(\gr{g})).$$

For the twisted cohomology $H^\bullet(M\times_G EG , H)$ we will also show its completeness by showing it at the level of the twisted complex. For this we just need to show that the induced map
$$\psi: \lim_{\leftarrow k} \Omega^\bullet(M_k) \to \lim_{\leftarrow p} \lim_{\leftarrow k} \Omega^\bullet(M_k)/F^p\Omega^\bullet(M_k)$$
is an isomorphism, where we have denoted $M_k=M\times_G EG_k$ to simplify the notation. As the filtration is exhaustive, then the map $\psi$ is injective. Now let us show that is surjective.

Any element $$\alpha \in \lim_{\leftarrow p} \lim_{\leftarrow k} \Omega^\bullet(M_k)/F^p\Omega^\bullet(M_k)$$ consists  of a sequence
$\alpha= \{\alpha_p \}_p $ where $\alpha_p \in \lim_{\leftarrow k} \Omega^\bullet(M_k)/F^p\Omega^\bullet(M_k)$ and $\alpha_{p+1} \mapsto \alpha_p$. Each $\alpha_p$ consists also of a sequence $\alpha_p=\{\alpha_{p,k}\}_k$ where $\alpha_{p,k} \in \Omega^\bullet(M_k)/F^p\Omega^\bullet(M_k)$ and $\alpha_{p,k+1} \mapsto \alpha_{p,k}$. Therefore we have that the $\alpha$'s satisfy
$$\xymatrix{
\alpha_{p+1,k+1} \ar@{|->}[r] \ar@{|->}[d] & \alpha_{p+1,k} \ar@{|->}[d] \\
\alpha_{p,k+1} \ar@{|->}[r]  & \alpha_{p,k}.
}
$$

Note that if $d(k)$ is the dimension of $M_k=M\times_GEG_k$ then for $p > d(k)$ we have that
$$\Omega^\bullet(M_k)/F^p\Omega^\bullet(M_k) = \Omega^\bullet(M_k),$$
and therefore for all $p>d(k)$, $\alpha_{p,k}=\alpha_{d(k)+1,k}$.

So, define $\beta_k :=\alpha_{d(k)+1,k}$ and consider the element $$\beta=\{\beta_k\}_k \in \lim_{\leftarrow k} \Omega^\bullet(M_k).$$
We claim that $\psi(\beta) = \alpha$. For this let us see that for $p$ fixed $\beta \mapsto \alpha_p$, and this easy to check because for  $d(k) <p$ then 
$\alpha_{p,k}=\alpha_{d(k)+1,k}=\beta_k$, and for $d(k) \geq p$ then $\beta_k =\alpha_{d(k)+1,k} \mapsto \alpha_{p,k}$.

Then as the map $\psi$ is an isomorphism the complex $\Omega^\bullet(M\times_G EG)$ is complete and therefore the twisted cohomology
$H^\bullet(M\times_G EG, H)$ is complete.

\end{proof}

We have the following useful corollaries:

\begin{cor}
If the equivariant cohomology class of twisting form $\bf H$ is zero, then
$$H^{0}(\widehat{\Omega}^\bullet_G(M), {\bf{H}}) \cong \prod_{i=0}^\infty H^{2i}_G(M) \ \ \ {\mbox{and}} \ \ \ 
H^{1}(\widehat{\Omega}^\bullet_G(M), {\bf{H}}) \cong \prod_{i=0}^\infty H^{2i+1}_G(M)$$
where $H^{*}_G(M)$ is the equivariant cohomology of $M$. 
\end{cor}
\begin{proof}
Because $\bf H$ is cohomologous to zero we  have that
$$H^{\bullet}(\widehat{\Omega}^\bullet_G(M), {\bf{H}})\cong H^{\bullet}(\widehat{\Omega}^\bullet_G(M), 0)$$
and by the Chern-Weill homomorphism we have that
$$H^{\bullet}(\widehat{\Omega}^\bullet_G(M), 0) \cong \lim_{\leftarrow k}H^\bullet(M\times_G EG_k, 0)$$
as the cohomology commutes with the inverse limit. \\
The spaces $M\times_G EG_k$ are finite dimensional manifolds, then their twisted cohomologies twisted by zero are isomorphic to their cohomologies but  $\integer_2$ graded, i.e.
$$H^{0}(M \times_G EG_k, 0) \cong \bigoplus_{i=0}^\infty H^{2i}(M\times_G EG_k) $$ $$ H^{1}(M \times_G EG_k, 0) \cong \bigoplus_{i=0}^\infty H^{2i+1}(M\times_G EG_k) .$$
Taking the inverse limit we have then
$$H^{0}(M\times_G EG,0) \cong \prod_{i=0}^\infty H^{2i}(M\times_G EG)$$
$$H^{1}(M\times_G EG,0) \cong \prod_{i=0}^\infty H^{2i+1}(M\times_G EG);$$
the result now follows from the Chern-Weil homomorphism for the untwisted case.
\end{proof}

\begin{cor} If $G$ acts freely on M, then the twisted equivariant
cohomology is isomorphic to the
 twisted cohomology of $M/G$.
\end{cor}
\begin{proof} This follows from the fact that twisted cohomology is a cohomology
theory (see \cite{Atiyah-Segal}) and the fact that $M\times_G EG $ and $M/G$ are homotopically equivalent. 
Therefore the
twisted cohomology of $M\times_G EG$  is
isomorphic to the twisted cohomology of $M/G$.
Note that in this case the twisted equivariant cohomology is finitely generated. This is because $M/G$ is a manifold and its twisted cohomology is finitely generated.  
\end{proof}

\begin{cor} \label{cohomology axioms}The twisted equivariant cohomology satisfies all the
axioms of a
 cohomology theory.
\end{cor}
\begin{proof} Because of the Chern-Weil isomorphism, the twisted equivariant
cohomology is isomorphicm to the twisted cohomology of the space $M\times_G EG$. Because the twisted cohomology satisfies all the axioms of
cohomology, the result follows.
\end{proof}

\begin{rem}
Corollary \ref{cohomology axioms} together with Theorem \ref{isomorphism to overline HH} imply that extended equivaraint cohomology for compact Lie groups satisfy the properties of a cohomology theory as are the Mayer-Vietoris long exact sequence, excision and the Thom isomorphism. This was also proved in \cite{HuUribe} using equivariant methods.
\end{rem}

\begin{example}
Let us calculate again the twisted equivariant cohomology of Example \ref{Example circle circle} using the twisted Chern-Weil homomorphism of Theorem \ref{Theorem Twisted Chern-Weil}.
Taking $S^{2k+1} \subset \complex^{k+1}$ as the set $ES^1_k$, then we have that $M \times_G EG_k = S^1 \times \complex P^k$.
The cohomology of $S^1 \times \complex P^k$ is 
$$H^*(S^1 \times \complex P^k)=\Lambda[d\theta] \otimes \real[\Omega]/\langle \Omega^{k+1} \rangle$$ and
the three form induced by the Chern-Weil map is $d\theta \Omega$.

As both manifolds $S^1$ and $\complex P^k$ are formal, the twisted cohomology could be calculated from the cohomology
of $S^1 \times \complex P^k$ and the operator $-d\theta \Omega \wedge$, i.e.
$$H_{S^1}^\bullet(S^1 \times \complex P^k; d\theta \Omega ) = H^\bullet \left( H^\bullet(S^1 \times \complex P^k; -d\theta \Omega \wedge )\right).$$
Then we have that
$$H_{S^1}^\bullet(S^1 \times \complex P^k; d\theta \Omega )= \real\langle d\theta \rangle \oplus \real \langle \Omega^k \rangle.$$

Now, if we take the inverse limit of these cohomologies we get
$$\lim_{\leftarrow k}  H_{S^1}^\bullet(S^1 \times \complex P^k; d\theta \Omega ) = \lim_{\leftarrow k} \real\langle d\theta \rangle \oplus \real \langle \Omega^k \rangle = \real\langle d\theta \rangle,$$
which implies that
$$H_{S^1}^\bullet(S^1 \times \complex P^\infty; d\theta \Omega ) =  \real\langle d\theta \rangle$$
thus agreeing with the calculations done in Example  \ref{Example circle circle}.
\end{example}

\section{Hamiltonian actions on generalized complex manifolds}\label{hamilton}

In this last section we will show how to induce a generalized complex structure on $M \times_G P$ whenever we have a Hamiltonian $G$ action on the 
generalized complex manifold $M$ and $P \to Q$ is a $G$-principal bundle over a generalized complex manifold $Q$. 

Let us start by recalling the definitions and theorems of generalized complex geometry that will be used in what follows (see \cite{Cavalcanti, Gualtieri, Bursztyn, Hu1, Stienon, Lin}).

\begin{definition}
A generalized complex manifold is a manifold $M$ together with one of the following equivalent structures:
\begin{itemize}
\item An endomorphism $\JJB : \TTB M \to \TTB M$ such that $\JJB^2=-1$, orthonormal with respect to the inner product $\langle , \rangle$ and such that
the $\sqrt{-1}$-eigenbundle  $L< \TTB M \otimes \complex$ is involutive with respect to the $H$-twisted Courant bracket i.e. $[L,L]_H \subset L$.
\item A maximal isotropic subbundle $L< \TTB M \otimes \complex$ which is involutive with respect to the $H$-twisted Courant bracket and such that $L \cap \overline{L}=\{0\}$.
\item A line bundle $U$ in $\wedge^* T^*M \otimes \complex$ generated locally by a form of the form $\rho = e^{B + \sqrt{-1}\omega} \Omega$ such that $\Omega$ is a decomposable complex form, $B$ and $\omega$ are real 2-forms and $\Omega \wedge \overline{\Omega} \wedge \omega^{n-k} \neq 0$
at the points where ${\rm{deg}}(\Omega)=k$; together with a section $\gr{X}=X + \xi \in \Gamma(\TTB M \otimes \complex)$ such that 
$$d_H \rho = \iota_{\gr{X}} \rho =\iota_X \rho + \xi \wedge \rho.$$
\end{itemize}
\end{definition}

The equivalence of these three definitions can be found in Gualtieri's thesis \cite{Gualtieri}. Let us only note that the elements of $\TTB M \otimes \complex$ that annihilate the form  $\rho$ of the third definition define the subbundle $L$. The line bundle $U$ is called the \emph{canonical line bundle} of the generalized complex structure.

\begin{definition}
We say that a Lie group $G$ acts on the generalized complex manifold $(M,\JJB)$ if the group $G$ acts by extended symmetries on $\TTB M$ and if the generalized complex structure $\JJB$ is preserved by the action.\\

The action is {\bf{Hamiltonian}} if there is an equivariant moment map $\mu: M \to \gr{g}^*$ such that $\JJB(d\mu_a)=X_a+\xi_a$ for all $a \in \gr{g}$ and $\mu_a(m) := \mu(m)(a)$.
\end{definition}

The condition on the moment map could be rephrased as
$$\iota_{\gr{X}_a} \rho =\iota_{X_a} \rho + \xi_a \wedge \rho = - \sqrt{-1} d \mu_a \wedge \rho$$
whenever $\rho$ is a local section of the canonical line bundle.

Let us now state the main theorem of this section.

\begin{theorem}
Let $G$ be a compact Lie group, $P \to Q$ a $G$-principal bundle such that the base $Q$ is compact and is endowed with a generalized complex structure,  and a Hamiltonian action of $G$ on the generalized complex manifold $M$ . Then the manifold $M \times_G P$ admits a generalized complex structure. 
\end{theorem}

\begin{proof}
Let us start by choosing the splitting of Lemma \ref{overline HH} for the manifold $M$ thus defining a $G$-invariant twisting form $H_1 \in \Omega^3(M)$
and a moment map $\sigma : M \to \gr{g}^*$ with $\JJB(d\sigma_a)=X_a + \xi_a$. Let $H_2 \in \Omega^3(Q)$ be the twisting form of $Q$.

Following Weinstein's construction in \cite{Weinstein} (cp. \cite[Thm. 6.10]{McDuff}), if we consider the cotangent vertical bundle of $P$
$$T^{*v}P := P \times_G T^*G,$$
then every connection 1-form  $A$ of $P$ induces an equivariant map $\phi^A : T^{*v}P \to T^*P$ such that the two form $\omega_A := (\phi^A)^*\omega_{\rm{can}} \in \Omega^2(T^{*v}P)$ is $G$ invariant and restricts to the canonical symplectic form of the fibers of the bundle $p:T^{*v}P \to Q$
 (the fibers are isomorphic to $T^*G$). Moreover, as the map $\mu_P : T^*P \to \gr{g}^*$, $ \mu_P(p, v^*) = -L_p^*v^*$ is Hamiltonian in the usual sense, with $L^*_p: T^*_pP \to \gr{g}^*$ the dual of the linear map $L_p: \gr{g} \to T_pP$, $L_p \xi=p \cdot \xi$ then the composition
 $$\mu=\mu_P \circ \phi^A : T^{*v}P \to \gr{g}^*$$
 is a moment map for the $G$ action on the fibers of $T^{*v}P$.

If $\rho$  is the local form defining the generalized complex structure in $Q$, it was shown in the proof of Theorem 2.2 of \cite{Cavalcanti} that as $Q$ is compact there exists an $\epsilon >0$ such that the local form
$$\overline{\rho}= e^{\sqrt{-1} \epsilon \omega_A} \wedge p^*\rho$$
defines a generalized complex structure on $T^{*v}P$ with twisting form $p^*H_2$. 

The action of $G$ on $T^{*v}P$ is also Hamiltonian with respect to the generalized complex structure defined by $\overline{\rho}$  and with moment map
$\mu^\epsilon : T^{*v}P \to \gr{g}^*$, $\mu^\epsilon(\cdot):=\epsilon \mu(\cdot)$; let us see this. Let $\gr{g} \to T(T^{*v}P), \ \ a \mapsto Y_a$, be the infinitesimal action; for the action to be Hamiltonian we need the following equation to be satisfied
$$\iota_{Y_a} \overline{\rho} = -\sqrt{-1} d\mu^\epsilon_a \wedge \overline{\rho},$$
which follows from the following set of equalities
\begin{eqnarray*}
\iota_{Y_a} \overline{\rho} &=& \iota_{Y_a} (e^{-\sqrt{-1}\epsilon \omega_A }\wedge p^* \rho)\\
&=& -\sqrt{-1} \epsilon (\iota_{Y_a} \omega_A)  \wedge e^{-\sqrt{-1}\epsilon \omega_A }\wedge p^* \rho\\
&=& -\sqrt{-1} \epsilon (d \mu_a)\wedge  \overline{\rho} \\
&=& -\sqrt{-1} d\mu^\epsilon_a \wedge \overline{\rho}.
\end{eqnarray*}
Notice that the equation $\iota_{Y_a} \omega_A=d \mu_a$ is the restriction to $T^{*v}P$ of the equation
$$\iota_{\phi^A_*Y_a}\omega_{\rm{can}} = d(\mu_P)_a$$ which follows from the fact that $\mu_P$ is a moment map.

\vspace{0.5cm}

We now have that the action of $G$ is Hamiltonian in both generalized complex manifolds $T^{*v}P$ and $M$. Then we can consider the product $M \times T^{*v}P$ together the generalized complex structure induced by $M$ and $T^{*v}P$ and whose twisting form is $H_1 + p^* H_2$. 
The diagonal action of $G$ on $M \times T^{*v}P$ is also Hamiltonian (see \cite[Prop. 3.9]{Hu1}) and its moment map is $\overline{\mu}= \sigma \oplus \mu$. Because $0$ is a regular value of $\mu_P$, and therefore of $\mu$, then $0$ is a regular value of $\overline{\mu}$. The group $G$ acts freely on $T^{*v}P$ and therefore it acts freely on $\overline{\mu}^{-1}(0)$. By the reduction theorem for Hamiltonian actions on generalized complex manifolds (see \cite{Bursztyn, Lin, Hu1, Stienon}) we have that the manifold $\overline{\mu}^{-1}(0)/G$ possesses a generalized complex structure. The manifold
$\overline{\mu}^{-1}(0)$ can be easily identified with $M\times P$ and therefore the quotient $M \times_GP$ becomes a generalized complex manifold.

The twisting form for the generalized complex structure on $M \times_GP$ (following \cite[Cor. 4.7]{Hu1}) is the basic 3-form
$$H_1 + p^* H_2 + d(\theta^a \xi_a) - \frac{1}{2}d( \theta^b \theta^c \iota_c \xi_b) \in \Omega^3(M \times P)_{bas}$$
where $\theta^a, \theta^b  \dots$ are the connection 1-forms for the principal $G$-bundle $M \times P \to M \times_GP$.
\end{proof}

\begin{rem} Let us finish by noting that in the case that the manifolds $BG_k$ are symplectic (for example when $G=U(n)$ and the $BG_k$'s are the complex grassmanians) then the manifolds $M\times_G EG_k$ would acquire a generalized complex structure with twisting form 
$$H_1 + d(\theta^a \xi_a) - \frac{1}{2} d(\theta^b \theta^c \iota_c \xi_b).$$
This three form is the same one we used in Theorem \ref{Theorem Twisted Chern-Weil} to construct the twisted Chern-Weil homomorphism.
\end{rem}

\section{Appendix} \label{appendix}
Let us consider the extended  action of the circle $S^1$ on the  manifold $S^1$ as in example \ref{Example circle circle}, and let us calculate the cohomology of the uncompleted complex
$\Omega^\bullet(S^1) \otimes \real[ \Omega] \ \ \mbox{with differential} \ \ d_{\gr{g}, \delta}=d - d\theta \Omega \wedge.$

The complex is therefore given by even forms $\sum_{i=0}^n f_i \Omega^i$ where $f_i \in C^\infty(S^1)$, and by odd forms
$\sum_{j=0}^m g_jd\theta \Omega^j$ with $g_j d\theta \in \Omega^1(S^1)$. It follows that the closed forms are
all the odd forms.

Now let us find the cohomology of the complex;  for this we need to understand which odd forms are cohomologous.
If we consider the equality
\begin{eqnarray} \label{cohomology relation twisted complex}
d_{\gr{g}, \delta} (f_j \Omega^j)= \frac{\partial f_j}{\partial \theta} d \theta \Omega^j - f^j d \theta \Omega^{j+1}
\end{eqnarray}
we can see that the following odd forms are all cohomologous:

\begin{eqnarray*}
\sum_{i=0}^{n}g_i d \theta \Omega^i & \simeq &
g_0 d\theta + g_1 d \theta \Omega + \cdots + g_{n-2} d\theta \Omega^{n-2} + \left(\frac{\partial g_n}{\partial \theta} + g_{n-1} \right) d\theta \Omega^{n-1}\\
& \simeq &
g_0 d\theta + g_1 d \theta \Omega + \cdots + g_{n-3} d\theta \Omega^{n-3} + \left(\frac{\partial^2 g_{n}}{\partial \theta^2} + \frac{\partial g_{n-1}}{\partial \theta} + g_{n-2} \right) d\theta \Omega^{n-2}\\
 & \simeq & \left( \frac{\partial^n g_n}{\partial \theta^n} + \frac{\partial^{n-1} g_{n-1}}{\partial \theta^{n-1}} + \cdots + \frac{\partial g_{1}}{\partial \theta}
+g_0 \right)d\theta \Omega^0
\end{eqnarray*}

So we can focus only on the odd forms that have only nonzero component in the coefficient of $\Omega^0$.
Let us see which odd forms in $\Omega^1(S^1) \cong C^\infty(S^1) \otimes_\real \real[d\theta]$ are exact. We have that
$$d_{\gr{g},\delta} \left( \sum_{j=0}^m f_j u^j\right) = g d  \theta$$ for $f_j$ and $g$ in $C^\infty(S^1)$. This equation implies the following set
of equalities:
\begin{eqnarray*}
df_0 &=& gd\theta\\
df_1 & = & f_0 d \theta  \\
& \cdots& \\
df_m &=& f_{m-1} d \theta\\
0 &=& f_m d \theta
\end{eqnarray*}
and if we check these equations starting from bottom to top, we see that $0=f_m=f_{m-1} = \cdots = f_0$ and
therefore none of the forms of the type $gd\theta$ are exact.
Then we can conclude that the cohomology of the uncompleted complex  is
equal to $$H^{0}(\Omega^\bullet(S^1) \otimes \real[ \Omega]; d_{\gr{g}, \delta})=0 \ \ \ \ \ \ \ H^{1}(\Omega^\bullet(S^1) \otimes \real[ \Omega]; d_{\gr{g}, \delta})\cong \Omega^1(S^1).$$

Let us see  what is the $H^*(BS^1)=\real[\Omega]$ module structure. We just need to check what happens with the forms
whose only non zero coefficient is the one of $\Omega^0$. Then $\Omega \cdot (gd\theta) = g d \theta \Omega$, and from equation (\ref{cohomology relation twisted complex})
we have that $\Omega \cdot g d \theta  \simeq \frac{\partial g }{\partial \theta} d \theta$. Then, the $\real[\Omega]$ module structure
on $H^{1}(\Omega^\bullet(S^1) \otimes \real[ \Omega]; d_{\gr{g}, \delta})\cong C^\infty(S^1) \otimes \real \langle d\theta \rangle $ is given by the operator $\frac{\partial}{\partial \theta}: C^\infty(S^1) \to C^\infty(S^1),$
\begin{eqnarray*}
H^{1}(\Omega^\bullet(S^1) \otimes \real[ \Omega]; d_{\gr{g}, \delta}) & \stackrel{\Omega\cdot}{\To} & H^{1}(\Omega^\bullet(S^1) \otimes \real[ \Omega]; d_{\gr{g}, \delta})\\
C^\infty(S^1) & \stackrel{\frac{\partial}{\partial \theta}}{\To} & C^\infty(S^1).
\end{eqnarray*}


It is easy to see now that the torsion submodule of $C^\infty(S^1)$ as a $\real[\Omega]$-module is infinitely generated (the functions $\sin(k\theta)$ belong to the torsion submodule). Therefore the cohomology $H^{\bullet}(\Omega^\bullet(S^1) \otimes \real[ \Omega]; d_{\gr{g}, \delta})$ is infinitely generated as a $\real[\Omega]$-module and this prevents this cohomology to have a topological meaning.

\vspace{0.5cm}

Also note that the completed algebra $\Omega^\bullet(S^1) \otimes \real[[ \Omega]]$ is isomorphic to the algebra 
$$\Omega^\bullet(S^1) \otimes \real[ \Omega] \otimes_{\real[ \Omega]} \real[[ \Omega]].$$
Then one may think that one can calculate the twisted equivariant cohomology by tensoring with ${\otimes}_{\real[ \Omega]} \real[[ \Omega]$ the cohomology of the uncompleted differential complex. This turns out to be false in general as one can check from the previous example:
the twisted equivariant cohomology is $H^\bullet_{S^1}(S^1,  d \theta\Omega)= \real$ meanwhile the cohomology of the uncompleted differential complex tensored with ${\otimes}_{\real[ \Omega]} \real[[ \Omega]$ is
$$H^{\bullet}(\Omega^\bullet(S^1) \otimes \real[ \Omega]; d_{\gr{g}, \delta}) {\otimes}_{\real[ \Omega]} \real[[ \Omega] \cong \Omega^1(S^1) {\otimes}_{\real[ \Omega]} \real[[ \Omega]$$
which is an infinitely generated $\real[[\Omega]]$-module.

\bibliographystyle{alpha}
\bibliography{Chern-Weil-twisted}
\end{document}